\documentclass{amsart}

\usepackage{amsmath, amssymb, amsthm, amscd, mathtools, amsfonts, enumerate, mathrsfs}

\theoremstyle{plain}
\newtheorem{theorem}{Theorem}[section]
\newtheorem{proposition}[theorem]{Proposition}
\newtheorem{lemma}[theorem]{Lemma}

\theoremstyle{definition}
\newtheorem{definition}[theorem]{Definition}

\theoremstyle{remark}
\newtheorem{remark}[theorem]{Remark}

\newcommand{\Bk}{(\mathcal{B}_k)}
\newcommand{\C}{(\mathcal{D}_1)}
\newcommand{\arcs}{\stackrel{\leftrightarrow}{a}\!} 

\numberwithin{equation}{section}

\begin{document}

\title[On dominating pair degree conditions for hamiltonicity]{On dominating pair degree conditions for hamiltonicity in balanced bipartite digraphs}
\author{Janusz Adamus}
\address{J.Adamus, Department of Mathematics, The University of Western Ontario, London, Ontario N6A 5B7 Canada}
\email{jadamus@uwo.ca}
\thanks{The research was partially supported by Natural Sciences and Engineering Research Council of Canada.}

\begin{abstract}
We prove several new sufficient conditions for hamiltonicity and bipancyclicity in balanced bipartite digraphs, in terms of sums of degrees over dominating or dominated pairs of vertices.
\end{abstract}
\maketitle


\section{Introduction}
\label{sec:intro}

This article is concerned with sufficient conditions for hamiltonicity and bipancyclicity in balanced bipartite digraphs. More specifically, we study several Meyniel-type criteria, that is, theorems asserting existence of hamiltonian cycles under certain conditions on the sums of degrees of non-adjacent vertices.
There are numerous such criteria, and open problems, in general digraphs (see, e.g., \cite{BG, BGL} and the references therein). Over the last few years, various analogues of these theorems and conjectures have been established for bipartite digraphs \cite{A1, A2, AAY, D, M, W, WW}. These results, generally speaking, do not follow from their non-bipartite analogues and require different arguments and techniques.

We begin with a short review of the relevant results to provide context for our present work. Throughout this paper, $D$ denotes a strongly connected balanced bipartite digraph of order $2a$ (see Section~\ref{sec:not} for details on notation and terminology).

The main feature of Meyniel-type criteria is that a degree condition be only imposed on pairs of non-adjacent vertices. The first such criterion in the bipartite setting was proved in \cite{AAY}.

\begin{theorem}[{\cite[Thm.\,1.2]{AAY}}]
\label{thm:AAY}
Let $D$ be as above, with $a\geq2$, and suppose that
\[
d(u)+d(v)\geq3a
\]
for every pair of distinct vertices $\{u,v\}$ such that $uv\notin A(D)$ and $vu\notin A(D)$. Then, $D$ is hamiltonian.
\end{theorem}

The lower bound of $3a$ is sharp (see examples in \cite{AAY}). The condition from Theorem~\ref{thm:AAY} may be further strengthened, in the spirit of~\cite{BGL}, by requiring that it be satisfied only by dominating and dominated pairs of vertices. This was done in \cite{A1}.

\begin{theorem}[{\cite[Thm.\,1]{A1}}]
\label{thm:A1}
Let $D$ be as above, with $a\geq3$, and suppose that
\[
d(u)+d(v)\geq3a
\]
whenever $\{u,v\}$ is a dominating or dominated pair. Then, $D$ is hamiltonian.
\end{theorem}

At this point, there are two natural questions: First, are the above assumptions enough to imply existence of cycles of all even lengths in $D$, perhaps modulo some exceptional digraphs (Bondy's metaconjecture)? And secondly, could we expect the same conclusion if the degree sum condition was only satisfied by the dominating pairs of vertices? The answer to the first question is positive. More precisely, we have the following result.

\begin{theorem}[{\cite[Thm.\,1.3]{A2}}]
\label{thm:A2}
Let $D$ be as above, with $a\geq3$, and suppose that
\[
d(u)+d(v)\geq3a
\]
whenever $\{u,v\}$ is a dominating or dominated pair. Then, $D$ is either bipancyclic or a directed cycle of length $2a$.
\end{theorem}

The second question seems much harder. However, most recently, Wang and Wu~\cite{WW} proposed an interesting variant of the degree sum condition that allows them to obtain hamiltonicity by only imposing the condition on dominating pairs of vertices. 

\begin{theorem}[{\cite[Thm.\,1.10]{WW}}]
\label{thm:WW}
Let $D$ be a strongly connected balanced bipartite digraph of order $2a$, where $a\geq3$, and let $k$ be an integer satisfying $\max\{1,\frac{a}{4}\}<k\leq\frac{a}{2}$.
Suppose that for every dominating pair $\{u,v\}$ of vertices in $D$,
\[
d(u)\geq2a-k\ \mathrm{and}\ d(v)\geq a+k,\quad\mathrm{or}\quad d(u)\geq a+k\ \mathrm{and}\ d(v)\geq2a-k.
\]
Then, $D$ is hamiltonian.
\end{theorem}

The authors of~\cite{WW} posed also several interesting problems related to the above theorems. Among them:
\begin{itemize}
\item[(a)] Are the assumptions of Theorem~\ref{thm:WW} enough to imply bipancyclicity of $D$?
\item[(b)] Is there an integer $k\geq0$ such that $D$ is hamiltonian if the inequality $d(u)+d(v)\geq3a+k$ is only imposed on the dominating pairs $\{u,v\}$?
\end{itemize}
The main goal of the present article is to prove the following positive answers to these two questions.

\begin{theorem}
\label{thm:bipart}
If $D$ satisfies the hypotheses of Theorem~\ref{thm:WW}, then $D$ is either bipancyclic or a directed cycle of length $2a$.
\end{theorem}

\begin{theorem}
\label{thm:hamil}
Let $D$ be a strongly connected balanced bipartite digraph of order $2a$, where $a\geq2$. Suppose that for every dominating pair $\{u,v\}$ of vertices in $D$,
\[
d(u)+d(v)\geq3a+1\,.
\]
Then, $D$ is hamiltonian.
\end{theorem}

Theorems~\ref{thm:bipart} and~\ref{thm:hamil} are proved in Sections~\ref{sec:bipart-proof} and~\ref{sec:hamil-proof}, respectively. In the last section, we discuss some corollaries and open problems.

\medskip

\section{Notation and terminology}
\label{sec:not}

We consider digraphs in the sense of \cite{BG}: A \emph{digraph} $D$ is a pair $(V(D),A(D))$, where $V(D)$ is a finite set (of \emph{vertices}) and $A(D)$ is a set of ordered pairs of distinct elements of $V(D)$, called \emph{arcs} (i.e., $D$ has no loops or multiple arcs).

The number of vertices $|V(D)|$ is the \emph{order} of $D$ (also denoted by $|D|$). For vertices $u$ and $v$ from $V(D)$, we write $uv\in A(D)$ to say that $A(D)$ contains the ordered pair $(u,v)$. If $uv\in A(D)$, then $u$ is called an \emph{in-neighbour} of $v$, and $v$ is an \emph{out-neighbour} of $u$. A pair of vertices $u,v\in V(D)$ is called \emph{dominating} (resp. \emph{dominated}) when there exists a vertex $w$ such that $uw\in A(D)$ and $vw\in A(D)$ (resp. $wu\in A(D)$ and $wv\in A(D)$). 

For vertex sets $S,T\subset V(D)$, denote by $A[S,T]$ the set of all arcs of $A(D)$ from a vertex in $S$ to a vertex in $T$. We define
${\arcs}(S,T)\coloneqq|A[S,T]|+|A[T,S]|$.

For a vertex set $S \subset V(D)$, we denote by $N^+(S)$ the set of vertices in $V(D)$ \emph{dominated} by the vertices of $S$; i.e.,
\[
N^+(S)=\{u\in V(D): vu\in A(D)\text{\ for\ some\ }v\in S\}\,.
\]
Similarly, $N^-(S)$ denotes the set of vertices of $V(D)$ \emph{dominating} vertices of $S$; i.e,
\[
N^-(S)=\{u\in V(D): uv\in A(D)\text{\ for\ some\ }v\in S\}\,.
\]
If $S=\{v\}$ is a single vertex, the cardinality of $N^+(\{v\})$ (resp. $N^-(\{v\})$), denoted by $d^+(v)$ (resp. $d^-(v)$) is called the
\emph{outdegree} (resp. \emph{indegree}) of $v$ in $D$. The \emph{degree} of $v$ is $d(v)\coloneqq d^+(v)+d^-(v)$.

More generally, for a vertex $v\in V(D)$ and a subdigraph $E$ of $D$, we will denote the cardinality of $N^+(\{v\})\cap V(E)$ by $d^+_E(v)$. Similarly, the cardinality of $N^-(\{v\})\cap V(E)$ will be denoted by $d^-_E(v)$. We set $d_E(v)\coloneqq d^+_E(v)+d^-_E(v)$.
We will denote by $E^c$ the subdigraph of $D$ spanned by the vertices $V(D)\setminus V(E)$. Consequently, $d^+_{E^c}(v)=|N^+(\{v\})\cap V(D)\setminus V(E)|$ and $d^-_{E^c}(v)=|N^-(\{v\})\cap V(D)\setminus V(E)|$.

A directed cycle (resp. directed path) on vertices $v_1,\dots,v_m$ in $D$ is denoted by $[v_1,\ldots,v_m]$ (resp. $(v_1,\dots,v_m)$). We will refer to them as simply \emph{cycles} and \emph{paths} (skipping the term ``directed''), since their non-directed counterparts are not considered in this article at all.
A cycle passing through all the vertices of $D$ is called \emph{hamiltonian}, or a \emph{Hamilton cycle}. A digraph containing a hamiltonian cycle is called a \emph{hamiltonian digraph}. A digraph containing cycles of all lengths is called \emph{pancyclic}.

A digraph $D$ is \emph{strongly connected} when, for every pair of vertices $u,v\in V(D)$, $D$ contains a path originating in $u$ and terminating in $v$ and a path originating in $v$ and terminating in $u$. A digraph $D$ in which, for every pair of vertices $u,v\in V(D)$ precisely one of the arcs $uv, vu$ belongs to $A(D)$ is called a \emph{tournament}.

A digraph $D$ is \emph{bipartite} when $V(D)$ is a disjoint union of independent sets $V_1$ and $V_2$ (the \emph{partite sets}).
It is called \emph{balanced} if $|V_1|=|V_2|$. One says that a bipartite digraph $D$ is \emph{complete} when $d(x)=2|V_2|$ for all $x\in V_1$. A complete bipartite digraph with partite sets of cardinalitites $a$ and $b$ will be denoted by $K^*_{a,b}$\,. A balanced bipartite digraph containing cycles of all even lengths is called \emph{bipancyclic}.
A \emph{matching} from $V_1$ to $V_2$ is an independent set of arcs with origin in $V_1$ and terminus in $V_2$ ($u_1u_2$ and $v_1v_2$ are \emph{independent} arcs when $u_1\neq v_1$ and $u_2\neq v_2$). If $D$ is balanced, one says that such a matching is \emph{perfect} if it consists of precisely $|V_1|$ arcs.

Finally, to streamline the proofs of Theorems~\ref{thm:bipart} and~\ref{thm:hamil}, we will use the following shorthand terminology (borrowed from~\cite{WW}).

\begin{definition}
\label{def:A-Bk-C}
Let $D$ be a balanced bipartite digraph of order $2a$.
For an integer $k\geq0$, we say that $D$ satisfies \emph{condition $\Bk$}, when every dominating pair $\{u,v\}$ satisfies
\[
d(u)\geq2a-k\ \mathrm{and}\ d(v)\geq a+k,\quad\mathrm{or}\quad d(u)\geq a+k\ \mathrm{and}\ d(v)\geq2a-k.
\]
Also, for $k\geq0$, we say that $D$ satisfies \emph{condition $(\mathcal{D}_k)$}, when every dominating pair $\{u,v\}$ satisfies
\[
d(u)+d(v)\geq 3a+k\,.
\]
\end{definition}

\medskip

\section{Proof of Theorem~\ref{thm:hamil}}
\label{sec:hamil-proof}

Throughout this section we assume that $D$ is a strongly connected balanced bipartite digraph with partite sets of cardinalities $a\geq2$, which satisfies condition $\C$.
The proof of Theorem~\ref{thm:hamil} is based on the following four simple lemmas.

\begin{lemma}
\label{lem:1}
Suppose that $D$ is non-hamiltonian. Then, for every vertex $u\in V(D)$ there exists a vertex $v\in V(D)\setminus\{u\}$ such that $\{u,v\}$ is a dominating pair.
\end{lemma}

\begin{proof}
For a proof by contradiction, suppose that $D$ contains a vertex $u_0$ which has no common out-neighbour with any other vertex in $D$. We claim that then no vertex of $D$ has a common out-neighbour with any other vertex. Indeed, let $v\in V(D)\setminus\{u_0\}$ be arbitrary. By strong connectedness of $D$, there is a path $P=(u_0,u_1,\ldots,u_s)$, with $u_s=v$. By assumptions on $u_0$, we have $d^-(u_1)=1$ and hence $d(u_1)\leq a+1$. If then $u_1$ had a common out-neighbour with some vertex $w\in V(D)$, we would have $d(w)\geq2a$, by condition $\C$. In particular, $w$ would be dominated by all the vertices from the opposite partite set, and so $u_0$ would have $w$ as a common out-neighbour with all vertices from its partite set; a contradiction. It thus follows that $u_1$ has no common out-neighbour with any other vertex in $D$. Repeating the above argument for all the subsequent vertices on $P$, we obtain in the end that $v$ has no common out-neighbour with any other vertex in $D$. This proves our claim, since $v$ was arbitary.

The strong connectedness now implies that $D$ is, in fact, a cycle of length $2a$. This contradicts the assumptions of the lemma.
\end{proof}

\begin{lemma}
\label{lem:2}
If $D$ is non-hamiltonian, then $d(u)\geq a+1$ for every vertex $u$ in $D$.
\end{lemma}

\begin{proof}
This follows immediately from Lemma~\ref{lem:1}, condition $\C$, and the fact that the degree of every vertex in $D$ is bounded above by $2a$.
\end{proof}

\begin{lemma}
\label{lem:3}
$D$ contains a cycle factor.
\end{lemma}

\begin{proof}
Suppose that $D$ is non-hamiltonian.
Let $X$ and $Y$ denote the two partite sets of $D$. Observe that $D$ contains a cycle factor if and only if there exist both a perfect matching from $X$ to $Y$ and a perfect matching from $Y$ to $X$. Therefore, by the K{\"o}nig-Hall theorem (see, e.g., \cite{B}), it suffices to show that $|N^+(S)|\geq|S|$ for every $S\subset X$ and $|N^+(T)|\geq|T|$ for every $T\subset Y$.

For a proof by contradiction, suppose that a non-empty set $S\subset X$ is such that $|N^+(S)|<|S|$. Then $|S|\geq2$, for else the sole vertex $x_0$ of $S$ would satisfy $d^+(x_0)=0$, which is not possible in a strongly connected digraph. Since $|N^+(S)|<|S|$, there exist vertices $x_1,x_2\in S$ with a common out-neighbour. By condition $\C$, we get
\[
3a+1\leq d(x_1)+d(x_2)=(d^+(x_1)+d^+(x_2))+(d^-(x_1)+d^-(x_2))\leq 2(|S|-1)+2a\,,
\]
and hence $2|S|\geq a+3$.

Now, for every $y\in Y\setminus N^+(S)$, we have $d(y)=d^+(y)+d^-(y)\leq a+(a-|S|)$. It follows that $|S|\leq a-1$, for else we would have $d(y)\leq a$, contrary to Lemma~\ref{lem:2}. Consequently, $|Y\setminus N^+(S)|\geq2$.
Moreover, no two vertices of $Y\setminus N^+(S)$ form a dominating pair. Indeed, for if $y_1,y_2\in Y\setminus N^+(S)$ were such a pair, we would have
\[
3a+1\leq d(y_1)+d(y_2)\leq 2(2a-|S|)\leq 4a-(a+3)\,,
\]
a contradiction. Thus, in fact, for every $y\in Y\setminus N^+(S)$, we have
\[
d^+(y)\leq a-(|Y\setminus N^+(S)|-1)=a-(a-|N^+(S)|-1)\leq|S|\,.
\]
Consequenly, for every such $y$,
\[
d(y)=d^+(y)+d^-(y)\leq|S|+(a-|S|)=a\,,
\]
which again contradicts Lemma~\ref{lem:2}.

This completes the proof of existence of a perfect matching from $X$ to $Y$. The proof for a matching in the opposite direction is analogous.
\end{proof}

We shall also need the following result from \cite{AAY}. Note that, by Lemma~\ref{lem:3}, $D$ contains a cycle factor.

\begin{lemma}[{\cite{AAY}}]
\label{lem:4}
Suppose that $D$ is non-hamiltonian, and let $\{C_1,\dots,C_l\}$ be a cycle factor in $D$ with a minimal number of elements.
Then,
\[
\arcs(V(C_1), V(D)\setminus V(C_1))\ \leq\ \frac{|C_1|(2a-|C_1|)}{2}\,.
\]
\end{lemma}

\medskip

\subsection*{Proof of Theorem~\ref{thm:hamil}}
\label{subsec:thm-hamil}

Let $D$ be a balanced bipartite digraph on $2a$ vertices, and let $X$ and $Y$ denote its partite sets. By Lemma~\ref{lem:3}, $D$ contains a cycle factor $\{C_1,\dots,C_l\}$. Assume $l$ is minimum possible, and for a proof by contradiction suppose that $l\geq2$. We may assume that $|C_1|\leq\dots\leq|C_l|$. Set $t\coloneqq|C_1|/2$. Then, 
$1\leq t\leq a/2$, since $l\geq2$ and $|C_1|\leq|C_2|$.
Moreover, by Lemma~\ref{lem:4}, we have
\begin{equation}
\label{eq:C1}
\arcs(V(C_1), V(D)\setminus V(C_1))\ \leq\ 2t(a-t)\,.
\end{equation}
Without loss of generality, we may assume that
\begin{equation}
\label{eq:C1X}
\arcs(V(C_1)\cap X, V(D)\setminus V(C_1))\ \leq\ t(a-t)\,,
\end{equation}
as otherwise
\begin{equation}
\label{eq:C1Y}
\arcs(V(C_1)\cap Y, V(D)\setminus V(C_1))\ \leq\ t(a-t)\,.
\end{equation}

We will first show that $t\geq2$. Suppose otherwise. Then $C_1$ is a 2-cycle consisting of, say, arcs $x_1y_1$ and $y_1x_1$. By \eqref{eq:C1},
\begin{multline}
\notag
d(x_1)+d(y_1)=(d_{C_1}(x_1)+d_{C_1}(y_1))+(d_{C^c_1}(x_1)+d_{C^c_1}(y_1))\\
\leq 4+2(a-1)=2a+2\,,
\end{multline}
which in light of Lemma~\ref{lem:2} implies that $d(x_1)=d(y_1)=a+1$, and $d_{C_1}(x_1)=d_{C_1}(y_1)=2$.
By Lemma~\ref{lem:1}, there exists a vertex $x'\in X\setminus\{x_1\}$ such that $\{x_1,x'\}$ is a dominating pair. Condition $\C$ then implies that $d(x')=2a$. In particular, $x'y_1\in A(D)$. We have $x'\in V(C_j)$ for some $1<j\leq l$. Let $y'$ denote the successor of $x'$ on the cycle $C_j$. Note that $\{y_1,y'\}$ form a dominating pair (as they both dominate $x'$), hence $d(y')=2a$, by condition $\C$ again. In particular, $x_1y'\in A(D)$, and so the cycle $C_1$ can be merged into $C_j$ by replacing the arc $x'y'$ on $C_j$ with the path $(x',y_1,x_1,y')$. This contradicts the minimality of $l$. Thus, indeed, $t\geq2$.

Let now $x_1,\dots,x_t\in X$ and $y_1,\dots,y_t\in Y$ be the vetices of $C_1$, labeled so that
\begin{equation}
\label{eq:C^c_1}
d_{C^c_1}(x_1)\leq\dots\leq d_{C^c_1}(x_t)\qquad\mathrm{and}\qquad d_{C^c_1}(y_1)\leq\dots\leq d_{C^c_1}(y_t)\,.
\end{equation}
Then, by \eqref{eq:C1X}, $d_{C^c_1}(x_1)\leq a-t$. The remainder of the proof splits into two cases depending on whether or not the latter inequality is strict.

\subsubsection*{Case 1.} Suppose first that $d_{C^c_1}(x_1)=a-t$. In this case we have $d_{C^c_1}(x_i)=a-t$ for all $1\leq i\leq t$, by \eqref{eq:C1X}. It follows that no two vertices $x_i,x_j$ in $X\cap V(C_1)$ form a dominating pair. Indeed, otherwise
\begin{multline}
\notag
3a+1\leq d(x_i)+d(x_j)=(d_{C_1}(x_i)+d_{C_1}(x_j))+(d_{C^c_1}(x_i)+d_{C^c_1}(x_j))\\ \leq 4t+2(a-t)=2a+2t\leq 3a\,,
\end{multline}
a contradiction.
Consequently,
\begin{equation}
\label{eq:y12}
d^-_{C_1}(y_j)=1 \quad\mathrm{for\ each\ } 1\leq j\leq t\,.
\end{equation}
In particular, $d^+_{C_1}(x_1)=1$. Now, as $d(x_1)\geq a+1$ (Lemma~\ref{lem:2}) and $d_{C^c_1}(x_1)=a-t$, it follows that $d^-_{C_1}(x_1)\geq t$ and so both $y_1$ and $y_2$ dominate $x_1$. However, by equality in \eqref{eq:C1X}, inequality \eqref{eq:C1} implies that \eqref{eq:C1Y} holds, hence (by \eqref{eq:C^c_1}) $d_{C^c_1}(y_1)+d_{C^c_1}(y_2)\leq2(a-t)$. This, together with \eqref{eq:y12} and condition $\C$, yields
\begin{multline}
\notag
3a+1\leq d(y_1)+d(y_2)=(d_{C_1}(y_1)+d_{C_1}(y_2))+(d_{C^c_1}(y_1)+d_{C^c_1}(y_2))\\ \leq 2(t+1)+2(a-t)=2a+2\,,
\end{multline}
which is impossible, as $a\geq2$.

\subsubsection*{Case 2.} Suppose then that $d_{C^c_1}(x_1)=a-t-\mu_1$ for some $\mu_1>0$. We claim that then $x_1$ forms a dominating pair with at least $\mu_1$ distinct vertices from $C_1$. Indeed, the inequality $d(x_1)\geq a+1$ implies that
\begin{equation}
\label{eq:x1}
d^+_{C_1}(x_1)\geq (a+1)-(a-t-\mu_1)-d^-_{C_1}(x_1)\geq 1+t+\mu_1-t=\mu_1+1\,,
\end{equation}
and so $x_1$ dominates at least $\mu_1$ vertices on $C_1$ apart from its own successor on $C_1$.

Note that, for any $1\leq i<j\leq t$ such that $x_i,x_j$ satisfy $d_{C^c_1}(x_i)+d_{C^c_1}(x_j)\leq2(a-t)$, the pair $\{x_i,x_j\}$ is not dominating. Indeed, for such $x_i$, $x_j$ we have
\begin{multline}
\notag
d(x_i)+d(x_j)=(d_{C_1}(x_i)+d_{C_1}(x_j))+(d_{C^c_1}(x_i)+d_{C^c_1}(x_j))\\ \leq 4t+2(a-t)=2a+2t\leq 3a\,.
\end{multline}
In particular, $\{x_1,x_2\}$ is not a dominating pair, by \eqref{eq:C1X} and \eqref{eq:C^c_1}.

Moreover, for all $x_i,x_j$ as above and for any vertices $x',x''$ such that $\{x_i,x'\}$ and $\{x_j,x''\}$ are dominating pairs (which exist, by Lemma~\ref{lem:1}, although not necessarily distinct), we have by condition $\C$
\begin{multline}
\notag
d(x')+d(x'')\\
\geq 6a+2-[(d^+_{C_1}(x_i)+d^+_{C_1}(x_j))+(d^-_{C_1}(x_i)+d^-_{C_1}(x_j))+(d_{C^c_1}(x_i)+d_{C^c_1}(x_j))]\\
\geq 6a+2-[t+2t+2(a-t)]=4a-t+2\,,
\end{multline}
hence $d(x')\geq (4a-t+2)-2a=2a-t+2$, and so
\begin{equation}
\label{eq:x'}
d_{C^c_1}(x')\geq (2a-t+2)-2t\geq a-t+2\,.
\end{equation}

Now, let $s\geq1$ and $\mu_1\geq\dots\geq\mu_s>0$ be such that $d_{C^c_1}(x_i)=a-t-\mu_i$ for all $1\leq i\leq s$, and $d_{C^c_1}(x_k)\geq a-t$ for $s<k\leq t$.
As in \eqref{eq:x1}, for each $1\leq i\leq s$, $x_i$ dominates at least $\mu_i$ vertices, say, $y^{i1},\dots,y^{i\mu_i}$ on $C_1$ apart from its own successor on $C_1$. Denote by $I_i$ the subset of $\{1,\dots,t\}$ of indices of the \emph{predecessors} on $C_1$ of those $y^{i1},\dots,y^{i\mu_i}$. Since no two $x_i,x_j$ form a dominating pair, we have
\[
(\{i\}\cup I_i)\cap(\{j\}\cup I_j)=\varnothing\quad\mathrm{for\ all\ } 1\leq i<j\leq s\,.
\]
Let $I:=\{1,\dots,t\}\setminus\bigcup_{i=1}^s(\{i\}\cup I_i)$. Then, by \eqref{eq:x'},
\begin{multline}
\notag
\sum_{i=1}^t d_{C^c_1}(x_i)=\sum_{i=1}^s\left(d_{C^c_1}(x_i)+\sum_{j\in I_i} d_{C^c_1}(x_j)\right)+\sum_{k\in I} d_{C^c_1}(x_k)\\
\geq \sum_{i=1}^s\left((a-t-\mu_i)+\mu_i\!\cdot\!(a-t+2)\right)+(t-\sum_{i=1}^s(1+\mu_i))\!\cdot\!(a-t)\\
= \sum_{i=1}^s\mu_i+\sum_{i=1}^t(a-t)> t(a-t)\,,
\end{multline}
which contradicts inequality \eqref{eq:C1X}. This completes the proof of the theorem.
\qed

\medskip

\section{Proof of Theorem~\ref{thm:bipart}}
\label{sec:bipart-proof}

The proof of Theorem~\ref{thm:bipart} is based on Theorems~\ref{thm:A2} and~\ref{thm:WW}, and the following result of Thomassen.

\begin{theorem}[{\cite[Thm.\,3.5]{T}}]
\label{thm:thomassen}
Let $G$ be a strongly connected digraph of order $n$, $n\geq3$, such that $d(u)+d(v)\geq 2n$ whenever $u$ and $v$ are non-adjacent. Then, $G$ is either pancyclic, or a tournament, or $n$ is even and $G$ is isomorphic to $K^*_{\frac{n}{2},\frac{n}{2}}$.
\end{theorem}

Let $a,k$ be integers such that $a\geq3$ and $\max\{1,\frac{a}{4}\}<k\leq\frac{a}{2}$. Note that then $a+k\leq2a-k$.
Throughout this section, we assume that $D$ is a strongly connected balanced bipartite digraph of order $2a$, which satisfies condition $\Bk$. By Theorem~\ref{thm:WW}, $D$ contains a Hamilton cycle $C$. Assume that $D$ is not equal to $C$. We shall show that then $D$ contains cycles of all even lengths.

\begin{lemma}
\label{lem:5}
For every vertex $u\in V(D)$ there exists a vertex $v\in V(D)\setminus\{u\}$ such that $\{u,v\}$ is a dominating pair.
In particular, $d(u)\geq a+k$.
\end{lemma}

\begin{proof}
For a proof by contradiction, suppose that $u_0\in V(D)$ has no common out-neighbour with any other vertex on $D$. Let $u_0^+$ denote the successor of $u_0$ on $C$. Then, $d^-(u_0^+)=1$ and so $d(u_0^+)\leq a+1$. Since $a+1<\min\{a+k,2a-k\}$, it follows that $u_0^+$ has no common out-neighbour with any other vertex, by condition $\Bk$. Repeating this argument for all the subsequent vertices along $C$, one gets that every vertex on $C$ is dominated in $D$ only by its predecessor on $C$. Thus, $D=C$, contrary to our assumptions.
The second claim follows immediately, by condition $\Bk$ and the fact that $a+k\leq 2a-k$.
\end{proof}

\begin{remark}
\label{rem:2-cycle}
Note that every vertex $u\in V(D)$ lies on a 2-cycle. Indeed, by Lemma~\ref{lem:5}, $d^+(u)+d^-(u)>a$ and hence $N^+(u)\cap N^-(u)\neq\varnothing$.
\end{remark}

\begin{lemma}
\label{lem:6}
Suppose $D$ is not bipancyclic. Then:
\begin{itemize}
\item[(a)] For every $u\in V(D)$,
\[
k+1\leq d^-(u)\leq a-1\quad\ \mathrm{and}\ \quad k+1\leq d^+(u)\leq a-1\,.
\]
\item[(b)] If $\{u,v\}$ is a non-dominating pair, with $u,v$ in the same partite set, then
\[
d(u)<2a-k\quad\mathrm{and}\quad d(v)<2a-k\,.
\]
\item[(c)] If $d(u)\geq2a-k$, then for any $v\in V(D)\setminus\{u\}$, 
\[
d^+(u)+d^-(v)\geq a+2\quad\mathrm{and}\quad d^+(v)+d^-(u)\geq a+2\,.
\]
\end{itemize}
\end{lemma}

\begin{proof}
Part (a) follows from Lemma~\ref{lem:5} and the fact that $D$ contains a Hamilton cycle.
For the proof of (b), suppose that $d(u)\geq2a-k$. Then, Lemma~\ref{lem:5} implies $d(u)+d(v)\geq(2a-k)+(a+k)=3a$. Hence, by part (a),
\[
d^+(u)+d^+(v)=(d(u)+d(v))-(d^-(u)+d^-(v))\geq3a-2(a-1)=a+2\,,
\]
and so $N^+(u)\cap N^+(v)\neq\varnothing$; a contradiction.

Finally, if $d(u)\geq2a-k$ then, by Lemma~\ref{lem:5}, $d(u)+d(v)\geq3a$ and hence
\[
d^+(u)+d^-(v)=(d(u)+d(v))-(d^-(u)+d^+(v))\geq3a-2(a-1)=a+2\,,
\]
by part (a), again. The other inequality of (c) is proved analogously.
\end{proof}

\subsection{Proof of Theorem~\ref{thm:bipart}}
\label{subsec:thm-bipart}

By Theorem~\ref{thm:A2}, we may assume that $D$ contains a dominated pair $\{u,v\}$ with $d(u)+d(v)<3a$, for else there is nothing to show. Then, by Lemma~\ref{lem:5}, we have $2(a+k)<3a$, hence $k<\frac{a}{2}$ and $a+k<2a-k$.

Let $X=\{x_1,\dots,x_a\}$ and $Y=\{y_1,\dots,y_a\}$ be the two partite sets of $D$, and suppose without loss of generality that $X$ contains a non-dominating pair $\{x',x''\}$, as above. Let $p\geq2$ denote the maximal integer such that $X$ contains vertices $\{x^{(1)},\dots,x^{(p)}\}$ no two of which form a dominating pair. Then, by Lemma~\ref{lem:6}(b), condition $\Bk$ implies that
\begin{equation}
\label{eq:non-p}
d(x)\geq2a-k\,,\quad\mathrm{for\ every\ }x\in X\setminus\{x^{(1)},\dots,x^{(p)}\}\,.
\end{equation}
Further, for any $1\leq i<j\leq p$, we have $N^+(x^{(i)})\cap N^+(x^{(j)})=\varnothing$, 
and hence
\begin{equation}
\label{eq:a-p-k}
a\geq p(k+1)\,,
\end{equation}
by Lemma~\ref{lem:6}(a). Since $k>\frac{a}{4}$, by assumption, it follows that $p\leq3$.
Moreover, when $p=3$, then $a\geq p(k+1)$ and $k>\frac{a}{4}$ imply $a>12$, and hence $k\geq4$. To sum up, we have
\begin{equation}
\label{eq:est}
p=2,\quad\mathrm{or\ else}\quad p=3\ \,\mathrm{and}\ \,k\geq4\,.
\end{equation}

For the remainder of this proof we shall assume that $k\geq3$. The case $k=2$ requires a different argument and it will be settled separately in Section~\ref{subsec:special-cases}.
We will reduce the proof to a straightforward application of Thomassen's Theorem~\ref{thm:thomassen}, by proving the following claim.

\subsubsection*{Claim 1.} $D$ is bipancyclic, else it contains a Hamilton cycle $C=[y_1,x_1,\dots,y_a,x_a]$ such that
\[
d^+(x_i)+d^-(y_i)\geq a+2\,,\qquad\mathrm{for\ every\ }1\leq i\leq a\,.
\]
For the proof of the claim, suppose first that $k=3$. Then, $p=2$, by~\eqref{eq:est}. Let $x',x''$ be the only two vertices in $X$ with degrees strictly less than $2a-k$.
By Lemma~\ref{lem:6}(c), it suffices to show that $D$ contains a Hamilton cycle $C$ such that $d^+(x')+d^-(x'^-)\geq a+2$ and $d^+(x'')+d^-(x''^-)\geq a+2$, where $x'^-$ (resp. $x''^-$) denotes the predecessor of $x'$ (resp. $x''$) on $C$.

Let then $C=[y_1,x_1,\dots,y_a,x_a]$ be a fixed Hamilton cycle in $D$, and suppose that the predecessor $x'^-$ of $x'$ has $d(x'^-)<2a-k$. Assume without loss of generality that $x'=x_1$ (and so $x'^-=y_1$).
To simplify notation, set
\[
\alpha\coloneqq d^+(x_1)\quad\mathrm{and}\quad \beta\coloneqq d^-(y_1)\,.
\]
We may assume that $\beta\leq a-3$, for if $\beta\geq a-2$, then Lemma~\ref{lem:6}(a) implies
\[
\alpha+\beta\geq(k+1)+(a-2)= a+2\,.
\]
Next, let
\[
l\coloneqq\min\{n\geq2:y_1x_n\in A(D), x_n\neq x''\}\,.
\]
By Lemma~\ref{lem:5}, we have $d^+(y_1)\geq(a+k)-\beta=a-(\beta-k)$, hence
\begin{equation}
\label{eq:l1}
l\leq\beta-k+3\,,
\end{equation}
where the inequality is strict unless $x''\in\{x_2,\dots,x_{\beta-k+2}\}$ and $y_1x''\in A(D)$.

Now, the pair $\{y_1,y_l\}$ is dominating (both dominate $x_l$), thus $d(y_l)\geq2a-k$ and
\begin{multline}
\label{eq:y_l-1}
|N^+(y_l)\setminus\{x_1,\dots,x_l\}|\geq d(y_l)-d^-(y_l)-l\\
\geq(2a-k)-(a-1)-l=a-k-l+1\,.
\end{multline}
On the other hand,
\begin{equation}
\label{eq:x_1-1}
|N^-(x_1)\setminus\{y_1,\dots,y_l\}|\geq d(x_1)-d^+(x_1)-l\geq(a+k)-\alpha-l\,.
\end{equation}
By~\eqref{eq:l1}, the right sides of inequalities~\eqref{eq:y_l-1} and~\eqref{eq:x_1-1} are positive, else $\alpha+\beta\geq a+3$.
It thus follows from~\eqref{eq:y_l-1} and~\eqref{eq:x_1-1} that there exists $l+1\leq m\leq a$ such that
\begin{equation}
\label{eq:m-1}
y_lx_m\in A(D)\ \mathrm{and}\ \,y_mx_1\in A(D)\quad(\mathrm{hence\ also\ } d(y_m)\geq2a-k)\,,
\end{equation}
unless
\begin{equation}
\label{eq:a-l}
(a-k-l+1)+(a+k-\alpha-l)\leq a-l\,.
\end{equation}
By~\eqref{eq:l1}, the latter inequality implies $\alpha+\beta\geq a+1$. Hence, either $\alpha+\beta\geq a+2$, or $\alpha+\beta=a+1$ and we have equalities in~\eqref{eq:a-l} and~\eqref{eq:l1}, or else there exists $l+1\leq m\leq a$ such that~\eqref{eq:m-1} holds. In the latter case, $D$ contains a Hamilton cycle
\[
C'=[y_1,x_l,\ldots,y_m,x_1,\ldots,y_l,x_m,\ldots,x_a]\,,
\]
where the dotted parts indicate the appropriate pieces of $C$. On this new cycle, the predecessor of $x'$ is of degree at least $2a-k$, and hence, by Lemma~\ref{lem:6}(c), we have decreased by one the number of pairs of vertices not satisfying the condition of Claim~1. In may still be the case that the predecessor of $x''$ on $C'$ has degree less than $2a-k$. If so, we repeat the above construction to replace $C'$ by another Hamilton cycle $C''$, on which the condition from Claim~1 is already satisfied by all pairs.

The equalities in~\eqref{eq:a-l} and~\eqref{eq:l1} can actually only occur when $D$ is bipancyclic, as is shown in Lemma~\ref{lem:nasty} below. We have thus proved Claim~1 in case $k=3$.
\medskip

Suppose than that $p=3$, and hence $k\geq4$. Let $x',x'',x'''$ be the only three vertices in $X$ with degrees strictly less than $2a-k$.
By Lemma~\ref{lem:6}(c) and the first part of the proof, it suffices to show that $D$ contains a Hamilton cycle $C$ such that $d^+(x')+d^-(x'^-)\geq a+2$, where $x'^-$ denotes the predecessor of $x'$ on $C$.

Let then $C=[y_1,x_1,\dots,y_a,x_a]$ be a fixed Hamilton cycle in $D$, and suppose that the predecessor $x'^-$ of $x'$ has $d(x'^-)<2a-k$. Assume without loss of generality that $x'=x_1$ (and so $x'^-=y_1$).
To simplify notation, set
\[
\alpha\coloneqq d^+(x_1)\quad\mathrm{and}\quad \beta\coloneqq d^-(y_1)\,.
\]
We may assume that $\beta\leq a-4$, for if $\beta\geq a-3$, then Lemma~\ref{lem:6}(a) implies
\[
\alpha+\beta\geq(k+1)+(a-3)\geq a+2\,.
\]
Next, let
\[
l\coloneqq\min\{n\geq2:y_1x_n\in A(D), x_n\notin\{x'',x'''\}\}\,.
\]
By Lemma~\ref{lem:5}, we have $d^+(y_1)\geq(a+k)-\beta=a-(\beta-k)$, hence
\begin{equation}
\label{eq:l2}
l\leq\beta-k+4\,,
\end{equation}
where the inequality is strict unless $x'',x'''\in\{x_2,\dots,x_{\beta-k+3}\}$, $y_1x''\in A(D)$, and $y_1x'''\in A(D)$.

As in the first part of the proof, the pair $\{y_1,y_l\}$ being dominating implies $d(y_l)\geq2a-k$, and so the inequality~\eqref{eq:y_l-1} holds. Of course, we have~\eqref{eq:x_1-1} as well.
By~\eqref{eq:l2}, the right sides of~\eqref{eq:y_l-1} and~\eqref{eq:x_1-1} are now positive, else $\alpha+\beta\geq a+4$.
It thus follows from~\eqref{eq:y_l-1} and~\eqref{eq:x_1-1} that there exists $l+1\leq m\leq a$ such that
\begin{equation}
\label{eq:m-2}
y_lx_m\in A(D)\ \mathrm{and}\ \,y_mx_1\in A(D)\quad(\mathrm{hence\ also\ } d(y_m)\geq2a-k)\,,
\end{equation}
unless~\eqref{eq:a-l} holds.
By~\eqref{eq:l2} and since $k\geq4$, the latter inequality implies $\alpha+\beta\geq a+1$. Hence, either $\alpha+\beta\geq a+2$, or $\alpha+\beta=a+1$ and we have equalities in~\eqref{eq:a-l} and~\eqref{eq:l2}, or else there exists $l+1\leq m\leq a$ such that~\eqref{eq:m-2} holds. In the latter case, as in the first part of the proof, $D$ contains a Hamilton cycle $C'$, on which the predecessor of $x'$ is of degree at least $2a-k$. Hence, we have reduced to the case $p=2$, which is already settled. On the other hand, the equalities in~\eqref{eq:a-l} and~\eqref{eq:l2} can only occur when $D$ is bipancyclic, by Lemma~\ref{lem:nasty} below, so the proof of Claim~1 is now complete.
\medskip

Suppose now that $D$ is not bipancyclic, and $C=[y_1,x_1,\dots,y_a,x_a]$ is a Hamilton cycle for which the condition of Claim~1 is satisfied.
We associate with $D$ a new digraph, $G$, constructed as follows. Set $V(G)\coloneqq\{v_1,\dots,v_a\}$, and $v_iv_j\in A(G)$ whenever $x_iy_j\in A(D)$, for $i,j\in\{1,\dots,a\}$, $i\neq j$. Then, $G$ is strongly connected because it contains a Hamilton cycle $[v_1,\ldots,v_a]$ (induced from $C$).

Note that $a\geq3$, so $G$ has at least three vertices. Moreover, for every $1\leq i\leq a$, we have
\begin{equation}
\label{eq:D-to-G}
d^+_G(v_i)\geq d^+_D(x_i)-1\quad\mathrm{and}\quad d^-_G(v_i)\geq d^-_{D}(y_i)-1\,.
\end{equation}
Therefore, by Claim~1, $d_G(v_i)\geq a$ for every $1\leq i\leq a$, and thus $G$ satisfies the assumptions of Theorem~\ref{thm:thomassen}.

Notice that every cycle $[v_{i_1},\dots,v_{i_l}]$ of length $l$ in $G$ corresponds to a cycle of length $2l$ in $D$, namely $[y_{i_1},x_{i_1},\dots,y_{i_l},x_{i_l}]$.
Also, by Remark~\ref{rem:2-cycle}, $D$ contains a cycle of length $2$.
To complete the proof it thus suffices to show that $G$ is not a tournament nor a balanced bipartite digraph.

If $G$ were a tournament, then it would contain no cycle of length 2, and hence $d_G(v_i)=d^+_G(v_i)+d^-_G(v_i)\leq a-1$ for every $i$; a contradiction.
On the other hand, if $1\leq i\leq a$ is such that $d_D(x_i)\geq 2a-k$ then $d^+_D(x_i)\geq a-k+1$, by Lemma~\ref{lem:6}(a), and hence $d^+_G(v_i)\geq a-k$, by~\eqref{eq:D-to-G}. Since $k<\frac{a}{2}$, it follows that $v_i$ dominates more than half of the vertices of $G$, and so $G$ is not balanced bipartite.\qed

\subsection{Special cases}
\label{subsec:special-cases}

There remain a few cases of digraphs not covered by the above proof. We do not know of any uniform way of tackling them all at once, and instead proceed on a case by case basis. We begin with a lemma that completes the proof of Theorem~\ref{thm:bipart} in the case of $k\geq3$.

\begin{lemma}
\label{lem:nasty}
Under the above notation, suppose that $p=2$, $k=3$, and we have equalities in~\eqref{eq:a-l} and~\eqref{eq:l1}, or else $p=3$, $k\geq4$, and we have equalities in~\eqref{eq:a-l} and~\eqref{eq:l2}. Then, $D$ is bipancyclic.
\end{lemma}

\begin{proof}
Let $C=[y_1,x_1,\ldots,y_a,x_a]$ be the fixed Hamilton cycle from the above proof.
The equality in~\eqref{eq:a-l} implies equalities in all the inequalities that led to it. In particular, $x_1$ is dominated by each of the $\{y_1,\dots,y_l\}$, and $d^-(y_l)=a-1$, so either $y_l$ is dominated by all the vertices from $X\setminus\{x_1\}$ or else it is dominated by all the vertices from $X\setminus\{x''\}$. The equality in~\eqref{eq:l1} (when $p=2$) or in~\eqref{eq:l2} (when $p=3$), in turn, implies that $x''=x_r$ for some $1<r<l$, and $y_1x_r\in A(D)$. Also, $l\geq4$.

Suppose first that $x_1y_l\notin A(D)$. Then, $D$ contains cycles of all even lengths (induced from $C$, by chords into $y_l$), except at most for the cycle $C_{2s}$ with $s=a-l+2$. Now, if $s\leq l-1$, then $D$ contains a $2s$-cycle $[x_1,\ldots,y_s]$ (where, as before, the dotted part indicates an appropriate piece of $C$).
If, in turn, $s\geq l$, then either the cycle $[y_1,x_r,\ldots,x_a]$ is of length $2s$, or else it is of length greater than $2s$ and we can shorten it to a $2s$-cycle by replacing an $x_q-y_l$ path on $C$ by the arc $x_qy_l$ with a suitable choice of $r\leq q<l-1$.

Suppose then that $x_ry_l\notin A(D)$. Then, $D$ contains cycles of all even lengths, except at most for the cycle $C_{2s}$ with $s=a-l+r+1$.
Now, if $2r\geq l$, then $D$ contains a $2s$-cycle
\[
[y_1,x_r,\ldots,y_{l-1},x_1,\ldots,x_{2r-l+1},y_l,\ldots,x_a]\,.
\]
If, in turn, $2r\leq l-1$, then $r<l-2$ (as $l\geq4$) and the cycle $[y_1,x_r,\dots,x_a]$ is of length $2(a-r+1)$, where
\[
2(a-r+1)\geq2(a-r+1)+2(2r-l+1)=2(a-l+r+2)\,.
\]
Thus, $[y_1,x_r,\dots,x_a]$ can be shortened to a $2s$-cycle by replacing an $x_q-y_l$ path on $C$ with the arc $x_qy_l$ with a suitable choice of $r<q<l-1$.
\end{proof}
\medskip

It remains to consider the case when $D$ is such that $k=p=2$. Fix a Hamilton cycle $C=[y_1,x_1,\dots,y_a,x_a]$ on $D$. By~\eqref{eq:a-p-k}, we have $a\geq6$. On the other hand, if $a\geq8$, then $k>\frac{a}{4}$ implies $k\geq3$. Thus, $a=6$ or $a=7$. In both cases, for every $x\in X\setminus\{x',x''\}$, we have $d(x)\geq2a-2$. Hence, for every such $x$,
\begin{equation}
\label{eq:a-1}
d^-(x)=d^+(x)=a-1\,,
\end{equation}
for else $D$ contains cycles of all even lengths through $x$ (induced from $C$). It is easy to see that then $D$ contains all cycles of even lengths less than $2(a-1)$. Indeed, let $2\leq s\leq a-2$ and suppose without loss of generality that $d^+(x_1)=d^-(x_1)=a-1$. If $y_{s+1}x_1\in A(D)$ then $D$ contains the $2s$-cycle $[x_1,\dots,y_{s+1}]$. If $x_1y_{a-s+2}\in A(D)$, then $D$ contains the $2s$-cycle $[x_1,y_{a-s+2},\dots,y_1]$. If, in turn, $y_{s+1}x_1\notin A(D)$ and $x_1y_{a-s+2}\notin A(D)$, then $D$ contains all other arcs adjacent to $x_1$, and hence it contains the $2s$-cycle $[x_1,y_{a-s+1},\dots,y_{s+2}]$. It thus remains to show that $D$ contains a subhamiltonian cycle $C_{2(a-1)}$.

Suppose first that $a=6$. Then, we have equality in~\eqref{eq:a-p-k}, and hence
\begin{equation}
\label{eq:a6}
d^+(x')=d^+(x'')=3,\quad\mathrm{and}\quad d^-(x')=d^-(x'')=a-1\,,
\end{equation}
by Lemma~\ref{lem:6}. We may assume that $d(x'^-)<2a-k$ or $d(x''^-)<2a-k$, for else $D$ satisfies the condition from Claim~1 in the proof in Section~\ref{subsec:thm-bipart}, and the remainder of that proof carries through.
Without loss of generality, assume then that $x'=x_1$ (hence $x'^-=y_1$) and $d(y_1)<2a-k$. We may further assume that $y_ax_1\notin A(D)$, for else $D$ contains the $2(a-1)$-cycle $[y_a,x_1,\dots,x_{a-1}]$. Then, by~\eqref{eq:a6}, we must have $y_{a-1}x_1\in A(D)$. Of the 3 vertices $\{x_2,\dots,x_{a-2}\}$ at most one doesn't satisfy \eqref{eq:a-1}. Similarly, at most one of the $\{y_2,\dots,y_{a-2}\}$ has degree less than $2a-2$, by Lemma~\ref{lem:6} and since $p=2$, $d(y_1)<2a-k$. Therefore, there exists $2\leq i\leq a-2$ such that $y_ix_{a-1},y_ax_i\in A(D)$, or $y_ix_a,y_1x_i\in A(D)$, or else $y_ix_{i+1}\in A(D)$. In either of the first two cases, we obtain a cycle of length $2(a-1)$ by replacing the arc $y_ix_i$ on the cycle $[y_{a-1},x_1,\dots,x_{a-2}]$ with an appropriate path of length 3. In the latter case, in turn, $D$ contains the $2(a-1)$-cycle $[x_{i+1},\ldots,y_i]$.

Finally, suppose $a=7$. Then, the inequalities $d(x')\geq a+2$, $d(x'')\geq a+2$, together with~\eqref{eq:a-p-k}, imply that
\begin{equation}
\label{eq:a7}
d^-(x')=a-1,\quad\mathrm{or\ else}\quad d^-(x'')=a-1\,.
\end{equation}
For if $d^-(x')\leq a-2$ and $d^-(x'')\leq a-2$, then $d^+(x')+d^+('')\geq8$ and hence $\{x',x''\}$ is a dominating pair; a contradiction.
Without loss of generality, assume then that $x'=x_1$ and $d^-(x_1)=a-1$. We may further assume that $y_ax_1\notin A(D)$, for else $D$ contains the $2(a-1)$-cycle $[y_a,x_1,\dots,x_{a-1}]$. Then, by~\eqref{eq:a7}, we must have $y_{a-1}x_1\in A(D)$. Of the 4 vertices $\{x_2,\dots,x_{a-2}\}$ at most one doesn't satisfy \eqref{eq:a-1}. Similarly, at most two of the 4 vertices $\{y_2,\dots,y_{a-2}\}$ have degree less than $2a-2$, by Lemma~\ref{lem:6} and since $p=2$. Therefore, there exists $2\leq i\leq a-2$ such that $y_ix_{a-1},y_ax_i\in A(D)$, or $y_ix_a,y_1x_i\in A(D)$, or else $y_ix_{i+1}\in A(D)$. In either of the first two cases, we obtain a cycle of length $2(a-1)$ by replacing the arc $y_ix_i$ on the cycle $[y_{a-1},x_1,\dots,x_{a-2}]$ with an appropriate path of length 3. In the latter case, in turn, $D$ contains the $2(a-1)$-cycle $[x_{i+1},\ldots,y_i]$.\qed

\medskip

\section{Final remarks}
\label{sec:rem}

First of all, for the sake of completeness, let us note that analogues of Theorems~\ref{thm:bipart} and~\ref{thm:hamil} for \emph{dominated} pairs hold true as well. More precisely, we have the following results.

\begin{theorem}
\label{thm:bipart2}
Let $D$ be a strongly connected balanced bipartite digraph of order $2a$, where $a\geq3$, and let $k$ be an integer satisfying $\max\{1,\frac{a}{4}\}<k\leq\frac{a}{2}$.
Suppose that for every dominated pair $\{u,v\}$ of vertices in $D$,
\[
d(u)\geq2a-k\ \mathrm{and}\ d(v)\geq a+k,\quad\mathrm{or}\quad d(u)\geq a+k\ \mathrm{and}\ d(v)\geq2a-k.
\]
Then, $D$ is either bipancyclic or a directed cycle of length $2a$.
\end{theorem}

\begin{theorem}
\label{thm:hamil2}
Let $D$ be a strongly connected balanced bipartite digraph of order $2a$, where $a\geq2$. Suppose that for every dominated pair $\{u,v\}$ of vertices in $D$,
\[
d(u)+d(v)\geq3a+1\,.
\]
Then, $D$ is hamiltonian.
\end{theorem}

Indeed, with a given digraph $D$ one can associate a digraph $D'$ on the same vertices, and with arcs $uv\in A(D')$ whenever $vu\in A(D)$. Then, for every $u\in V(D')=V(D)$, we have $d^-_{D'}(u)=d^+_D(u)$ and $d^+_{D'}(u)=d^-_D(u)$, hence $d_{D'}(u)=d_D(u)$. Moreover, a pair of vertices $\{u,v\}$ is dominating in $D'$ if and only if it is dominated in $D$. The above results thus follow immediately from Theorems~\ref{thm:bipart} and~\ref{thm:hamil}, by observing that any cycle $[u_1,\dots,u_s]$ in $D'$ corresponds to a cycle on the same vertices in $D$ traversed in the opposite direction, $[u_s,\dots,u_1]$.
\smallskip

Next, let us have a look at the lower bound on the integer $k$ in Theorems~\ref{thm:WW} and~\ref{thm:bipart}. Of course, the assumption that $k>\frac{a}{4}$ leaves out nearly half of the possible cases in condition $\Bk$. However, a careful analysis of the proof from Section~\ref{subsec:thm-bipart} shows that this argument carries through, for $a$ large enough, so long as $k$ is bounded below by a constant of the form $\lambda a$ for some real $\lambda>0$. Consequently, for all but finitely many exceptional digraphs satifying condition $\Bk$ of this type, hamiltonicity implies bipancyclicity:

\begin{proposition}
\label{prop:what-if}
For every $\lambda>0$ there exists $a_0$ such that the following holds:\\
If $D$ is a strongly connected hamiltonian balanced bipartite digraph of order $2a$, with $a\geq a_0$, $k$ is an integer satisfying $\max\{1,\lambda a\}<k\leq\frac{a}{2}$, and for every dominating pair $\{u,v\}$ of vertices in $D$,
\[
d(u)\geq2a-k\ \mathrm{and}\ d(v)\geq a+k,\quad\mathrm{or}\quad d(u)\geq a+k\ \mathrm{and}\ d(v)\geq2a-k\,,
\]
then $D$ is bipancyclic or else a directed cycle of length $2a$.
\end{proposition}

It would be therefore very interesting to know whether condition $\Bk$ with $\max\{1,\lambda a\}<k\leq\frac{a}{2}$ and $\lambda<\frac{1}{4}$ does imply hamiltonicity of a digraph. We do not know the answer to this question, and it seems implausible that the techniques of \cite{WW} could be used to obtain such a result.
\smallskip

Another interesting problem is this: Is condition $\C$ enough to imply bipancyclicity of a strongly connected balanced bipartite digraph, modulo some exceptional cases?
Again, the techniques used in the present paper seem insufficient to have a go at it.

Finally, in the context of Theorem~\ref{thm:hamil}, it is perhaps worth mentioning that our present methods can be pushed a bit further, and the conclusions of Lemmas~\ref{lem:1} and~\ref{lem:3} hold, in fact, even if a digraph $D$ satisfies just $(\mathcal{D}_0)$. We do not include the proofs here, because they are much less elegant and require considering several special cases. They do however give hope that an analogue of Theorem~\ref{thm:hamil} with condition $\C$ replaced by $(\mathcal{D}_0)$ could be true.



\medskip

\end{document}